\documentclass{amsart}

\usepackage{lmodern}
\usepackage{microtype}
\usepackage[english]{babel}
\usepackage{mathtools}
\usepackage{hyperref}

\theoremstyle{plain} 
\newtheorem*{theorem}{Theorem}
\newtheorem*{lemma}{Lemma}
\DeclareMathOperator{\mre}{Re} 
\DeclareMathOperator{\Arg}{Arg}

\begin{document} 
\title[The norm of the backward shift on $H^1$]{The norm of the backward shift on $H^1$ is $\frac{2}{\sqrt{3}}$} 
\date{\today} 

\author{Ole Fredrik Brevig} 
\address{Department of Mathematics, University of Oslo, 0851 Oslo, Norway} 
\email{obrevig@math.uio.no}

\author{Kristian Seip} 
\address{Department of Mathematical Sciences, Norwegian University of Science and Technology (NTNU), 7491 Trondheim, Norway} 
\email{kristian.seip@ntnu.no}

\thanks{Research supported in part by Grant 275113 of the Research Council of Norway.}

\dedicatory{Dedicated to Fritz Gesztesy on the occasion of his 70th birthday}

\begin{abstract} 
	We show that the norm of the backward shift operator on $H^1$ is $2/\sqrt{3}$, and we identify the functions for which the norm is attained.
\end{abstract}

\subjclass{Primary 30H10. Secondary 47B38.}

\maketitle

\section{Introduction}
Let $H^p$ denote the Hardy spaces of the unit disc $\mathbb D$ which are studied in depth in the classical monographs of Duren \cite{Duren1970} and Garnett \cite{Garnett2007}. The purpose of this note is to compute the norm of the backward shift operator
\[Bf(z) \coloneq \frac{f(z)-f(0)}{z}\]
acting on $H^1$. 

\begin{theorem}
	We have
	\[\|Bf\|_1 \leq \frac{2}{\sqrt{3}} \|f\|_1\]
	for every $f$ in $H^1$, and equality holds in this bound if and only if
	\[f(z) = C \left(\frac{\sqrt{3}+I(z)}{\sqrt{3}-I(z)}\right)^2,\]
	where $C$ is a constant and $I$ is an inner function satisfying $I(0)=0$.
\end{theorem}

Here we use the standard terminology that $I$ in $H^1$ is an inner function if $|I(e^{i\theta})|=1$ for almost every $e^{i\theta}$ on $\mathbb{T}$. We refer to \cite{Duren1970} and \cite{Garnett2007} for basic information about inner functions and the few other rudimentary aspects of the theory of $H^p$ required in this note.

The backward shift has been studied in great detail (see e.g.~\cite{CR2000}), but its norm is to the best of our knowledge only known on $H^2$ and $H^\infty$. Before proceeding with the proof of the theorem, we will briefly summarize the state-of-the-art. Let therefore $\|B\|_p$ stand for the norm of the backward shift operator on $H^p$.

It is plain that $\|B\|_2=1$, since 
\[\|Bf\|_2 = \sqrt{\|f\|_2^2 -|f(0)|^2}\]
for every $f$ in $H^2$ by orthogonality. We will make use of orthogonality in a similar manner below. 

For general $1 \leq p \leq \infty$, the estimate $|f(0)| \leq \|f\|_p$ and the triangle inequality provide the upper bound $\|B\|_p \leq 2$. A consequence \cite[Theorem~7.7]{BKS1988} is that $\|B\|_p > 1$ for any $p\neq2$. An elementary direct proof of the same assertion follows from the example in \cite[Lemma~2.3]{BOCS2021}.

The previous best result on $\|B\|_1$ is due to Ferguson \cite[Theorem~2.4]{Ferguson2017}, who established that $\|B\|_1 \leq 1.7047$. Ferguson \cite[Theorem~2.5]{Ferguson2017} also proved that $\|B\|_\infty = 2$, by observing that if $f$ is the conformal automorphism of $\mathbb{D}$ interchanging the origin and the point $w$, then $\|Bf\|_\infty = 1+|w|$. In addition, Ferguson improved the upper bound to $\|B\|_p \leq 2^{|1-2/p|}$ using Riesz--Thorin interpolation in $L^p(\mathbb{T})$.

The problem of determining $\|B\|_p$ has been raised in connection with Toeplitz operators (see \cite[Section~5]{Shargorodsky2021} and \cite[Open Problem~5.4]{KS2023}). We hope this note may inspire further work for $p\neq 1,2,\infty$, including the case $0<p<1$.

\section{Proof}
We define square-roots of $H^1$ functions as follows. We set $f^{1/2}(z)\coloneq 0$ if $z$ is a point at which $f(z)=0$, and otherwise we set
\[ f^{1/2}(z)\coloneq \sqrt{|f(z)|}e^{i \frac{\Arg{f(z)}}{2}}; \] 
here and elsewhere, $\sqrt{a}$ signifies the nonnegative square-root of a nonnegative number $a$, and $\Arg{w}$ is the principal value of the argument of the complex number $w$. We see that $f^{1/2}$ is defined at every point in $\mathbb{D}$ and almost everywhere on $\mathbb{T}$. We will need the following result about such square-roots.

\begin{lemma}
	If $f$ is a function in $H^1$ with $f(0) \geq 0$, then
	\begin{equation} \label{eq:subh} \sqrt{f(0)} \leq \mre{\int_0^{2\pi} f^{1/2}(e^{i\theta})\,\frac{d\theta}{2\pi}}.\end{equation}
	Equality in \eqref{eq:subh} is attained if and only if $f^{1/2}$ is analytic in $\mathbb{D}$. 
\end{lemma}

\begin{proof}
	We notice that the function 
	\[u(z) \coloneq \mre{f^{1/2}(z)}\]
	is continuous and in fact subharmonic in $\mathbb{D}$. Indeed, $u$ clearly satisfies the sub-mean value property at points $z$ where $u(z)=0$, and at all other points $u$ is locally harmonic. Hence
	\begin{equation} \label{eq:subharm}
		\sqrt{f(0)} = u(0) \leq \int_0^{2\pi} u(r e^{i\theta})\,\frac{d\theta}{2\pi}
	\end{equation}
	for $0<r<1$. Using that $|\mre{w_1}-\mre{w_2}| \leq \sqrt{|w_1^2-w_2^2|}$ for complex numbers $w_1$ and $w_2$ with nonnegative real part in combination with the Cauchy--Schwarz inequality, we find that
	\[\int_0^{2\pi} \big|u(r e^{i\theta})-\mre{f^{1/2}(e^{i\theta})}\big|\,\frac{d\theta}{2\pi} \leq \sqrt{\int_0^{2\pi} \left|f(r e^{i\theta})-f(e^{i\theta})\right|\,\frac{d\theta}{2\pi}}.\]
	Since $f$ is in $H^1$, the right-hand side goes to $0$ as $r \to 1^-$. The asserted inequality follows from this and \eqref{eq:subharm}. 
	
	If equality is attained in \eqref{eq:subh}, then the subharmonicity of $u$ means that equality is attained in \eqref{eq:subharm} for every $0<r<1$. This implies that $u$ is harmonic in $\mathbb{D}$. If $f$ is nontrivial, then $u$ is strictly positive in $\mathbb{D}$. Consequently, $f^{1/2}$ is analytic in $\mathbb{D}$.
\end{proof}

Let $f$ be a function in $H^1$. Then $f = IF$, where $I$ is an inner function and
\begin{equation} \label{eq:outerrep}
	F(z) = \exp\left(\int_0^{2\pi} \frac{e^{i\theta}+z}{e^{i\theta}-z}\,\log|f(e^{i\theta})|\,\frac{d\theta}{2\pi}\right).
\end{equation}
We say that $f$ is outer if $f=F$.

\begin{proof}[Proof of the theorem]
	Let $f$ be a function in $H^1$. We may assume without loss of generality that $f(0)\geq 0$. We set $a \coloneq \sqrt{|f(0)|}$ and note that
	\[f(e^{i\theta})-f(0) = \big(f^{1/2}(e^{i\theta})-a\big)\big(f^{1/2}(e^{i\theta})+a\big)\]
	for almost every $e^{i\theta}$ on $\mathbb{T}$. Setting next
	\[b \coloneq \int_0^{2\pi} f^{1/2}(e^{i\theta})\,\frac{d\theta}{2\pi},\]
	we get 
	\begin{multline*}
		\|Bf\|_1 \leq \sqrt{\|f^{1/2}\|_2^2 - |b|^2 + |a-b|^2}\sqrt{\|f^{1/2}\|_2^2 - |b|^2 + |a+b|^2} \\
		= \sqrt{\big(\|f\|_1+|a|^2\big)^2 - \big(2a\mre{b}\big)^2}
	\end{multline*}
	by the Cauchy--Schwarz-inequality and orthogonality. Using that $a\geq0$ and that $\mre{b} \geq a$ from the lemma, we obtain that
	\[\|Bf\|_1 \leq \|f\|_1 \sqrt{(1+x)^2-4x^2}\]
	for $x = f(0)/\|f\|_1\le 1$. The maximum of the right-hand side is attained for $x=1/3$, which completes the proof of the asserted inequality. 
	
	Suppose next that
	\begin{equation} \label{eq:attained}
		\|Bf\|_1 = \frac{2}{\sqrt{3}}\|f\|_1
	\end{equation}
	for a nontrivial function $f$ in $H^1$. Since plainly $f(0)\neq0$, we may assume without loss of generality that $f(0)>0$. Inspecting the argument above, we see that \eqref{eq:attained} can only hold if $\mre{b}=a$, which by the lemma means that $f^{1/2}$ is analytic. Moreover, we must have attained equality in our application of the Cauchy--Schwarz inequality. This is only possible if there is a constant $\lambda \geq 0$ such that
	\begin{equation} \label{eq:CSattained}
		\big|f^{1/2}(e^{i\theta})-a\big| = \lambda \big|f^{1/2}(e^{i\theta})+a\big|
	\end{equation}
	for almost every $e^{i\theta}$ on $\mathbb{T}$. Since $\mre{f^{1/2}}\geq0$ by definition and since $f(0)>0$ by assumption, we find that $\mre{f^{1/2}(z)}+a \geq a > 0$ for every $z$ in $\mathbb{D}$. This means that $f^{1/2}+a$ is an outer function, so the combination of \eqref{eq:outerrep} and \eqref{eq:CSattained} yields 
	\[f^{1/2}(z) - a = \lambda I(z) \big(f^{1/2}(z)+a\big)\]
	for some inner function $I$. Since the left-hand side vanishes at $z=0$, it is clear that $I(0)=0$. Moreover, since $|f^{1/2}(e^{i\theta}) - a|<|f^{1/2}(e^{i\theta})+a|$ for almost every $e^{i\theta}$ on $\mathbb{T}$, we must have $0<\lambda<1$. Consequently, 
	\[f(z) = f(0) \left(\frac{1+\lambda I(z)}{1-\lambda I(z)}\right)^2.\]
	A direct computation using that $(I^k)_{k\geq0}$ is an orthonormal set in $H^2$ shows that
	\[\frac{f(0)}{\|f\|_1} = \frac{1-\lambda^2}{1+3\lambda^2}.\]
	Since \eqref{eq:attained} is attained, the left-hand side equals $1/3$ which means that $\lambda = 1/\sqrt{3}$.
\end{proof}

\bibliographystyle{amsplain} 
\bibliography{backwardH1}

\providecommand{\bysame}{\leavevmode\hbox to3em{\hrulefill}\thinspace}
\providecommand{\MR}{\relax\ifhmode\unskip\space\fi MR }
% \MRhref is called by the amsart/book/proc definition of \MR.
\providecommand{\MRhref}[2]{%
  \href{http://www.ams.org/mathscinet-getitem?mr=#1}{#2}
}
\providecommand{\href}[2]{#2}
\begin{thebibliography}{1}

\bibitem{BKS1988}
A.~B\"{o}ttcher, N.~Krupnik, and B.~Silbermann, \emph{A general look at local
  principles with special emphasis on the norm computation aspect}, Integral
  Equations Operator Theory \textbf{11} (1988), no.~4, 455--479. \MR{950512}

\bibitem{BOCS2021}
O.~F. Brevig, J.~Ortega-Cerd\`a, and K.~Seip, \emph{Idempotent {F}ourier
  multipliers acting contractively on {$H^p$} spaces}, Geom. Funct. Anal.
  \textbf{31} (2021), no.~6, 1377--1413. \MR{4386412}

\bibitem{CR2000}
J.~A. Cima and W.~T. Ross, \emph{The {B}ackward {S}hift on the {H}ardy
  {S}pace}, Mathematical Surveys and Monographs, vol.~79, American Mathematical
  Society, Providence, RI, 2000. \MR{1761913}

\bibitem{Duren1970}
P.~L. Duren, \emph{Theory of {$H\sp{p}$} {S}paces}, Pure and Applied
  Mathematics, vol. Vol. 38, Academic Press, New York-London, 1970. \MR{268655}

\bibitem{Ferguson2017}
T.~Ferguson, \emph{Bounds on the norm of the backward shift and related
  operators in {H}ardy and {B}ergman spaces}, Illinois J. Math. \textbf{61}
  (2017), no.~1-2, 81--96. \MR{3770837}

\bibitem{Garnett2007}
J.~B. Garnett, \emph{Bounded {A}nalytic {F}unctions}, first ed., Graduate Texts
  in Mathematics, vol. 236, Springer, New York, 2007. \MR{2261424}

\bibitem{KS2023}
O.~Karlovych and E.~Shargorodsky, \emph{Bounded compact and dual compact
  approximation properties of {H}ardy spaces: new results and open problems},
  arXiv:2308.04072.

\bibitem{Shargorodsky2021}
E.~Shargorodsky, \emph{On the essential norms of {T}oeplitz operators with
  continuous symbols}, J. Funct. Anal. \textbf{280} (2021), no.~2, Paper No.
  108835, 11. \MR{4169308}

\end{thebibliography}

\end{document}